\documentclass[A4paper,reqno,oneside,12pt]{amsart}
\usepackage{fullpage}
\usepackage{setspace}
\usepackage{datetime}
\usepackage{tikz}
\usepackage{mathtools}
\usetikzlibrary{er}
\usetikzlibrary{arrows.meta}
\usepackage{amsmath,amsfonts,amssymb}
\usepackage{verbatim,enumitem,graphics}
\usepackage{microtype}
\usepackage{xcolor}
\usepackage[backref=page]{hyperref}
\hypersetup{%
    colorlinks,
    linkcolor={red!50!black},
    citecolor={green!50!black},
    urlcolor={blue!50!black}
}
\usepackage{dsfont}
\usepackage[abbrev,msc-links,backrefs]{amsrefs} 
\usepackage{doi}
\usepackage{lineno}
\usepackage[normalem]{ulem}

\newtheorem{theorem}{Theorem}[section]

\newtheorem{proposition}[theorem]{Proposition}

\theoremstyle{definition}
\newtheorem{definition}[theorem]{Definition}

\newcommand{\FF}{\mathbb F}
\newcommand{\cF}{\mathcal{F}}
\newcommand{\cH}{\mathcal{H}}
\newcommand{\cG}{\mathcal{G}}
\newcommand{\cL}{\mathcal{L}}

\newcommand{\cY}{\mathcal{Y}}
\newcommand{\cZ}{\mathcal{Z}}

\newcommand{\sqbinom}[2]{\begin{bsmallmatrix} #1 \\ #2 \end{bsmallmatrix}}

\newcommand{\lanran}[1]{\left \langle #1 \right\rangle}
\newcommand{\rank}{\text{rank}}

\DeclareMathOperator{\SP}{sp}
\let\epsilon=\varepsilon
\let\phi=\varphi

\begin{document}

\setstretch{1.27}
%\linenumbers

\title{A note on the induced Ramsey theorem for spaces.}

\author{Bryce Frederickson}

\author{Vojtech R\"{o}dl}

\author{Marcelo Sales}

\thanks{The second and third authors were partially supported by NSF grant DMS 1764385. The second author was also supported by NSF grant DMS 2300347 and the third author was partially supported by US Air Force grant FA9550-23-1-0298.}

\address{Department of Mathematics, Emory University, Atlanta, GA, USA}

\email{\{bfrede4|vrodl\}@emory.edu}

\address{Department of Mathematics, University of California, Irvine, CA, USA}

\email{mtsales@uci.edu}

\dedicatory{Dedicated to the 80th birthday of Bruce Lee Rothschild}

\begin{abstract}
The aim of this note is to give a simplified proof of the induced version of the Ramsey theorem for vector spaces first proved by H. J. Pr\"{o}mel \cite{Pr86}.
\end{abstract}

\maketitle

\section{Introduction}

We consider a Ramsey type result for vector and affine spaces. Let $q$ be a prime power and $\FF_q$ the finite field of order $q$. For our exposition, we will call a vector space or an affine space (translation of a vector space) over $\FF_q$ by \textbf{space}. A space $U$ of rank $k$ will be called a \textbf{$k$-space}. For a space $U$ of rank at least $k$, we denote by $\sqbinom{U}{k}$ the set of all $k$-subspaces of $U$. By the relation
\begin{align*}
    N \rightarrow  (n)_r^{k,\SP}
\end{align*}
we denote the fact that for any coloring of the $k$-subspaces of an $N$-space $V$ with $r$ colors, there exists a $n$-subspace $U \subseteq V$ such that $\sqbinom{U}{k}$ is monochromatic, i.e., any $k$-subspace of $U$ has the same color.

%Given integers $k,n,r\geq 1$, Ramsey's theorem \cite{R29} states that every sufficiently large integer $N$ satisfies the relation
%\begin{align*}
%    N\rightarrow (n)_r^k
%\end{align*}
%meaning that for every coloring of the $k$-tuples of $[N]=\{1,\ldots,N\}$ with $r$ colors, there exists a set $X\subseteq [N]$ of size $n$ such that all its $k$-tuples are monochromatic. 

In \cite{Rota70}, Rota conjectured that a version of Ramsey's theorem \cite{R29} holds for vector spaces. Graham and Rothschild proved in \cite{GR71param} a Ramsey theorem for $n$-parameter sets, which in particular implies the $1$-dimensional case of Rota's conjecture (Corollary 2, \cite{GR71param}). Together with Leeb, they fully settled the conjecture in \cite{GLR72} by proving a more general Ramsey theorem for a class of categories that includes structures as vector spaces, affine spaces and projective spaces. A simpler proof was given later by Spencer \cite{Sp79}.

%version for vector space over finite fields. Let $q$ be a power of prime and $\FF_q$ the finite field of order $q$. Given a vector space $\FF_q^n$, a subspace $V$ is a \textbf{$k$-space} of $\FF_q^n$ if $V$ is a $k$-dimensional vector space of $\FF_q^n$. Given a vector space $U$ over $\FF_q$ of dimension greater than $k$, we denote $\sqbinom{U}{k}$ as the set of all subspaces of $U$ of dimension $k$. We say that
%\begin{align*}
%    \FF_q^N \rightarrow (\FF_q^n)_r^k
%\end{align*}
%if for any coloring of the $k$-spaces of $\FF_q^N$ with $r$ colors, there exists a subspace $U \subseteq \FF_q^N$ of dimension $n$ such that $\sqbinom{U}{k}$ is monochromatic. Hence, Rota's conjecture can be formulated as finding sufficiently large $N$ such that the arrow relation above holds.

\begin{theorem}[\cites{GLR72, Sp79}]\label{th:vramsey}
For integers $k,n, r$ with $0\leq k \leq n$, there exists $N_0:=N_0(k,n,r)$ such that if $N\geq N_0$, then $N\rightarrow (n)^{k,\SP}_r$ holds.
\end{theorem}

Note that Theorem \ref{th:vramsey} holds for both vector spaces and affine spaces, though the minimum value of $N_0$ may depend on the case. Our goal in this note is to give a short proof of an induced version of Theorem \ref{th:vramsey}. Such versions were considered for hypergraphs in \cites{NR77,NR82, NR83, AH78, VR16}. The following theorem is due to Pr\"{o}mel.

\begin{theorem}[\cite{Pr86}]\label{th:indvramsey}
 Let $0\leq k \leq n$ and $r$ be positive integers, let $C$ be an $n$-space and $\cF \subseteq \sqbinom{C}{k}$ be a family of $k$-subspaces of $C$. Then there exists $N$, an $N$-space $X$ and a family $\cH\subseteq \sqbinom{X}{k}$ of $k$-subspaces of $X$ such that for any coloring of $\cH$ with $r$ colors, there exists an $n$-space $U\in \sqbinom{X}{n}$ with $\sqbinom{U}{k}\cap \cH$  monochromatic and isomorphic to $\cF$ (see Definition \ref{def:isomorphic}). 
\end{theorem}

The original proof of Theorem \ref{th:indvramsey} given in \cite{Pr86} is not simple. The proof was later simplified in \cite{FGR87} where the authors based their argument in a partite amalgamation construction, a technique that was used before to obtain similar results for graphs and hypergraphs \cites{NR82}. We further simplify the proof by avoiding the use of a partite amalgamation. In fact, our proof will be based only on Theorem \ref{th:vramsey} and the Hales--Jewett theorem \cite{HJ63}.

\section{Preliminaries}

The proof of Theorem \ref{th:indvramsey} is written in a way that holds both for the vector and affine space versions. In this section we define some common notation and state some common properties for both types of spaces. 

In a vector space $V$, a linear combination $\sum_{i = 1}^n \alpha_i v_i$ with $\sum_{i = 1}^n \alpha_i = 1$ is called an \textbf{affine combination}. A subset $U$ of $V$ is called an \textbf{affine space} if $U$ is closed under affine combinations. All affine subspaces of $V$ are of the form $U = v_0 + W$, where $v_0 \in V$ and $W$ is a linear subspace of $V$.

We now list a number of properties which are shared by vector and affine spaces. For the remainder of this note, we use general language to refer to both vector spaces and affine spaces simultaneously. For example, \textbf{space} will refer to either a vector space or an affine space, and \textbf{combination} will refer to a linear combination in a vector space context or an affine combination in an affine space context.

A subset $W$ of $U$ which is closed under combinations is called a \textbf{subspace} of $U$. A map $\phi : U \rightarrow V$ between spaces $U$ and $V$ which preserves combinations is called a \textbf{homomorphism}. A bijective homomorphism is called an \textbf{isomorphism}. A subset $B$ of a space $U$ generates a subspace $\lanran{B}$ of $U$ via combinations. If every element of $\lanran{B}$ has a unique representation as a combination of elements of $B$, we say $B$ is \textbf{independent}. We call $B$ a \textbf{basis} for $U$ if $B$ is independent and $\lanran{B} = U$. The size of a basis for $U$ is an invariant of $U$, called its \textbf{rank}, which we denote $\rank(U)$. Two spaces are isomorphic iff they have the same rank. We call a space of rank $k$ a \textbf{$k$-space}. We denote the set of all $k$-subspaces of $U$ by $\sqbinom{U}{k}$. Perhaps less standard is the following definition, pertinent to Theorem \ref{th:indvramsey}.
\begin{definition}\label{def:isomorphic}
	Given two spaces $V$ and $\tilde V$, we say that a family $\mathcal F \subseteq \sqbinom{V}{k}$ of $k$-subspaces of $V$ is \textbf{isomorphic} to a family $\tilde{\mathcal F} \subseteq \sqbinom{\tilde V}{k}$ of $k$-subspaces of $\tilde V$ if there exist subspaces $U \subseteq V$ and $\tilde U \subseteq \tilde V$ and an isomorphism $\varphi : U \to \tilde U$ such that $\mathcal F \subseteq \sqbinom{U}{k}$ and $\tilde{\mathcal F} \subseteq \sqbinom{\tilde U}{k}$, and $\varphi$ induces a bijection between $\mathcal F$ and $\tilde{\mathcal F}$.
\end{definition}

An important property for us is the following: If a space $U$ has basis $B$, then for every space $V$, every map $\phi : B \rightarrow V$ extends to a unique homomorphism $\tilde \phi: U \rightarrow V$. Similarly, for a collection $\{U_i\}_{i = 1}^\ell$ of spaces with respective finite disjoint bases $\{B_i\}_{i = 1}^\ell$ with $\bigcup_{i = 1}^\ell B_i$ independent, we define the \textbf{direct sum} $\bigoplus_{i = 1}^\ell U_i$ of $\{U_i\}_{i=1}^\ell$ to be the space $\lanran{\bigcup_{i = 1}^\ell B_i}$. The direct sum has the property that for every space $V$, every collection $\{\varphi_i : U_i \rightarrow V\}_{i = 1}^\ell$ of homomorphisms extends to a unique homomorphism $\varphi : \bigoplus_{i = 1}^\ell U_i \rightarrow V$ satisfying $\varphi \vert_{U_i} = \varphi_i$ for each $i = 1, \ldots, \ell$.

The last property that we mention is that any independent set $B_0$ in $U$ can be extended to a basis $B = B_0 \sqcup B_1$ for $U$, and in that case, $U = \lanran{B_0} \oplus \lanran{B_1}$. Consequently, every subspace $W$ of $U$ has a complementary subspace $W^c$ of $U$ such that $U = W \oplus W^c$.

\section{Proof of Theorem \ref{th:indvramsey}}

    %Our proof setup is similar to the initial step of the partite amalgamation construction of \cites{NR77, NR83}.
    
    Let $N_0:=N_0(k,n,r)$ be the integer given by Theorem \ref{th:vramsey} such that $N\rightarrow (n)^{k,\SP}_r$ for $N \geq N_0$, and let $E$ be an $N_0$-space. For each $n$-space $U \in \sqbinom{E}{n}$, consider an isomorphism $\phi_U:C\rightarrow U$. The isomorphism $\phi_U$ induces a copy $\cF_U$ of $\cF$ in $U$ given by
    \begin{align*}
        \cF_U:=\{\phi_U(F):\: F\in \cF\}.
    \end{align*}
    For each $k$-space $F \in \cF_U$, there exists an $(N_0 - k)$-subspace $F^{c} \subseteq E$ such that $E=F\oplus F^{c}$. We now consider a space $V$ of rank $\left|\sqbinom{E}{n}\right|(n + (N_0 - k)|\cF|)$ with basis $B_V$, and we partition $B_V$ as
		\[B_V = \bigcup_{U \in \sqbinom{E}{n}} \left(B_U \cup \bigcup_{F \in \mathcal \cF_U} B_{F^c}\right),\]
		where each $B_U$ is of size $n$, and each $B_{F^c}$ is of size $N_0 - k$. Now for each $U \in \sqbinom{E}{n}$ and for each $F \in \cF_U$, we have copies $W_U := \lanran{B_U}$ of $U$ and $W_{F^c} := \lanran{B_{F^c}}$ of $F^c$ such that
    \begin{align*}
        V = \bigoplus_{U \in \sqbinom{E}{n}} \left(W_U \oplus \bigoplus_{F \in \mathcal \cF_U} W_{F^{c}}\right).
    \end{align*}
    %That is, $V$ is the vector space where each $U$ has their own private set of $n$ coordinates and for each $F \in \cF_U$, the space $F^{c}$ has their own private set of $N_0-k$ coordinates. Consequently, 
    %\begin{align*}
    %    \dim (V)=\left(n+(N_0-k)|\cF|\right)\left| \sqbinom{\FF_q^{N_0}}{n}\right|.
    %\end{align*}
    
    Note that by construction there exists an isomorphism $\pi_U : W_U \rightarrow U$ for each $n$-subspace $U$ of $E$. The isomorphism $\pi_U$ creates a copy $\cF_{W_U}$ of $\cF_U$ in $W_U$ consisting of those $k$-subspaces $W_F := \pi_U^{-1}(F)$ which are mapped to some $F \in \cF_U$. Let $\cG$ be the collection of $k$-subspaces of $V$ given by 
    \begin{align*}
        \cG:=\bigcup_{U\in \sqbinom{E}{n}}\bigcup_{F\in\cF_U}\sqbinom{W_F\oplus W_{F^{c}}}{k},
    \end{align*}
    i.e. $\cG$ consists of all of the copies $\cF_{W_U}$ of $\cF$, with each $W_F \in \cF_{W_U}$ completed to its own private copy $\sqbinom{W_F \oplus W_{F^c}}{k}$ of $\sqbinom{E}{k}$.

	 Consider the homomorphism $\pi:V\rightarrow E$ defined in the following way. For each $U \in \sqbinom{E}{n}$, we have the isomorphism $\pi_U : W_U \to U \subseteq E$. Similarly, for each $F \in \cF_U$, we have an isomorphism $\pi_{F^c} : W_{F^c}\rightarrow F^c \subseteq E$. We define $\pi : V \rightarrow E$ to be the unique homomorphic extension of the maps $\pi_U$ and $\pi_{F^c}$ to $V$; that is, $\pi \vert_{W_U} = \pi_U$ for each $U \in \sqbinom{E}{n}$, and $\pi \vert_{W_{F^c}} = \pi_{F^c}$ for each $F \in \cF_U$.
   %\begin{align}\label{eq:1}
    %   \sqbinom{W_U}{k}\cap \cG\cong \cF_U\cong \cF.
   %\end{align}
	 Let
     \begin{align*}
         \cY=\left\{W_F\oplus W_{F^c}:\: U\in \sqbinom{E}{n},\, F\in \cF_U\right\}
     \end{align*}
    be the set of $N_0$-spaces covering $\cG$. Note that by construction $\pi$ induces an isomorphism from each $Y\in \cY$ to $E$. In particular, it identifies each $k$-space in $\cG$ with a $k$-subspace of $E$. Thus one can view $\pi$ as a projection of the collection $\cY$ onto the $N_0$-space $E$.
    
    We will apply the Hales--Jewett theorem \cite{HJ63} to the alphabet $\cY$. For that, we first establish some notation. Given an alphabet $A=\{a_1,\ldots,a_t\}$, we say that an element $S=(s_1,\ldots,s_N)$ of $A^N$ is a \textbf{word} of length $N$. A collection $\{S_1, S_2, \ldots, S_t\}$ of $t$ words of length $N$ with $S_i=(s_{i,1},\ldots,s_{i,N})$ is a \textbf{combinatorial line} if there exists a partition $[N]=I_M\cup I_F$, $I_M\neq \emptyset$, and a sequence $\{b_j\}_{j\in I_F}$ of elements of $A$ such that
    \begin{align*}
        s_{i,j}=\begin{cases}
        a_i, &\quad \text{ for $j\in I_M$}\\
        b_j, &\quad \text{ for $j\in I_F$}
        \end{cases}
    \end{align*}
    for $1\leq i \leq t$ and $1\leq j \leq N$. We will refer to $I_M$ as the moving part and $I_F$ as the fixed part. The Hales--Jewett theorem asserts that given integers $t,\ell\geq 1$, there exists integer $N:=N(t,\ell)$ such that the following holds. For any alphabet $A$ of size $t$ and any $\ell$-coloring of the set of words $A^N$, there exists a monochromatic combinatorial line $\cL\subseteq A^{N}$. 
    Let $N_1$ be the integer obtained by the Hales--Jewett theorem for $t=|\cY|$ and $\ell=r^m$, where $m=\left|\sqbinom{E}{k}\right|$.
    
    We will now construct our collection of $k$-spaces $\cH$. Let $X \subseteq V^{N_1}$ be the set defined by
    \begin{align*}
        X=\left\{(x_1,\ldots,x_{N_1})\in V^{N_1}:\: \pi(x_1)=\pi(x_2)=\ldots=\pi(x_{N_1})\right\},
    \end{align*}
    i.e., $X$ is the set of points in $V^{N_1}$ with same $\pi$-projection over $E$. It is not difficult to check that since $\pi$ is a homomorphism, the set $X$ is a space of rank $N\leq\rank(V^{N_1})$. We extend the map $\pi: V\rightarrow E$ to a map $\tilde{\pi}:X\rightarrow E$ by taking as image the common value of $\pi$ through all coordinates, i.e., 
    \begin{align*}
    \tilde{\pi}(x_1,\ldots,x_n)=\pi(x_1).    
    \end{align*}
    
    For any $p \leq N_0$, given $p$-subspaces $A_1,\ldots, A_{N_1} \subseteq V$ with $\pi(A_1) = \cdots = \pi(A_{N_1})$ and $\pi$ injective on each $A_i$, we denote
    \begin{align*}
    (A_1,\ldots,A_{N_1}):=\{(x_1,\ldots,x_{N_1}) \in X:\: x_i\in A_i,\, \pi(x_1) = \cdots = \pi(x_{N_1})\}   
    \end{align*}
    as the $p$-subspace of $X$ with elements in $A_1\times\ldots\times A_{N_1}$. Let $\cH$ be the collection of $k$-subspaces of $X$ given as follows:
    \begin{align*}
        \cH=\left\{(G_1,\ldots,G_{N_1}):\: G_i\in \cG, \, \pi(G_1)=\pi(G_2)=\ldots=\pi(G_{N_1})\right\}.
    \end{align*}
    We claim that for any $r$-coloring of $\cH$, there exists a $n$-space $U$ such that $\sqbinom{U}{k}\cap \cH$ is monochromatic and isomorphic to $\cF$.
    
    To check that, we need some preparation. Note that every $N_1$-tuple $(Y_1,\ldots,Y_{N_1})\in \cY^{N_1}$ is an $N_0$-subspace of $X$ since $\pi(Y_1)=\ldots=\pi(Y_{N_1})=E$. Hence, a combinatorial line of $\cY^{N_1}$ is a collection of $|\cY|$ distinct $N_0$-subspaces of $X$. In fact, the next proposition shows that this collection is isomorphic to $\cY$ as a collection of $N_0$-spaces.
    
    \begin{proposition}\label{prop:prod}
    Let $\cZ\subseteq \cY^{N_1}$ be a combinatorial line of $\cY^{N_1}$. Then there exists a space $\tilde{V}\subseteq X$ and an isomorphism $\phi: \tilde{V}\rightarrow V$ such that $\cZ$ is a collection of $N_0$-spaces of $\tilde{V}$ and $\phi$ induces a bijection between elements of $\cZ$ and $\cY$. Moreover, the map $\phi$ also satisfies $\pi \circ \phi = \tilde \pi \vert_{\tilde V}$ and induces a bijection between the $k$-spaces of $\sqbinom{\tilde{V}}{k}\cap \cH$ and the $k$-spaces of $\cG$, i.e., $\sqbinom{\tilde{V}}{k}\cap \cH\cong \cG$. 
    \end{proposition}
    
    \begin{proof}
    Suppose without loss of generality that the fixed part of $\cZ$ consists of the first $f$ indices and the moving part consists of the remaining $N_1-f$ indices, i.e.,
    \begin{align*}
        \cZ=\left\{(Z_1,\ldots,Z_f,Y,\ldots,Y):\: Y\in \cY \right\}
    \end{align*} 
    for fixed elements $Z_1, \ldots, Z_f$ of $\cY$. Define $\tilde{V}\subseteq X$ by
    \begin{align*}
        \tilde{V}=\left\{(z_1,\ldots,z_f,y,\ldots,y) \in X:\: z_i\in Z_i,\, y\in V,\, \pi(z_i) = \pi(y)\right\},
    \end{align*}
     and the map $\phi:\tilde{V}\rightarrow V$ by 
     \begin{align*}
     \phi(z_1,\ldots,z_f,y,\ldots,y)=y.
     \end{align*}
     Clearly $\tilde{V}$ is a space and $\phi$ is a homomorphism. Note that since $\pi$ is an isomorphism between each $Z_i$ and $E$, given $y \in V$, there exists a unique $z_i\in Z_i$ such that $\pi(z_i)=y$. Hence, $\phi$ is an isomorphism. Also, by construction, $\phi$ sends the $N_0$-space $(Z_1,\ldots,Z_f,Y,\ldots,Y)\in \cZ$ to $Y\in \cY$, giving our desired bijection.
     
     %Note that by definition $\cG=\bigcup_{Y\in \cY}\sqbinom{Y}{k}$. Therefore, every $k$-space of $\cG$ is contained in some unique $Y\in \cY$. If $(G_1,\ldots,G_{N_1})\in \sqbinom{\tilde{V}}{k}\cap \cH$, then by definition of $\cH$ we obtain that $G_i \in \sqbinom{Z_i}{k}$ for $1\leq i \leq t$ and $\tilde{G}:=G_{t+1}=\ldots=G_{N_1}\in \cG$. In particular, there exists $\tilde{Y} \in \cY$ such that $\tilde{G}\in \sqbinom{\tilde{Y}}{k}$ and consequently $\phi(G_1,\ldots,G_{N_1})=\phi(G_1,\ldots,G_t,\tilde{G},\ldots,\tilde{G})=\tilde{G} \in \cG$. Hence, $\phi$ sends $k$-spaces of $\sqbinom{\tilde{V}}{k}\cap \cH$ to the $k$-spaces of $\cG$. It only remains to check that for every $G\in \cG$ there exists unique $(G_1,\ldots,G_{N_1})\in \sqbinom{\tilde{V}}{k}\cap \cH$ such that $\phi(G_1,\ldots,G_{N_1})=G$. However, this follows from the fact that $G\in \sqbinom{Y}{k}$ for some unique $Y\in \cY$. Since $\pi$ is an isomorphism between each $Y\in \cY$ and $E$, then for $1\leq i \leq t$ there exists a unique $k$-space $G'_i\in \sqbinom{Z_i}{k}$ such that $\pi(G'_i)=\pi(G)$. Thus we can take $(G_1,\ldots,G_{N_1})$ with $G_i=G'_i$ for $1\leq i \leq t$ and $G_i=G$ for $t+1\leq i \leq N_1$.
		 Next, note that 
		 \[\pi(\phi(z_1, \ldots, z_f, y, \ldots, y)) = \pi(y) = \tilde \pi(z_1, \ldots, z_f, y, \ldots, y)\]
		 for every $(z_1, \ldots, z_f, y, \ldots, y) \in \tilde V$.
		 
		 Now for every $(G_1, \ldots, G_{N_1}) \in \sqbinom{\tilde V}{k} \cap \cH$, we have that $\phi(G_1, \ldots, G_{N_1}) = G_{N_1} \in \cG$ from the definition of $\phi$. Conversely, we can write $\cG = \bigcup_{Y \in \cY} \sqbinom{Y}{k}$, so each $G \in \cG$ is contained in some $Y \in \cY$. Since $\pi$ maps each of $Y, Z_1, \ldots, Z_f \in \cY$ isomorphically to $E$, there exists a unique $G_i$ in each $\sqbinom{Z_i}{k}$ such that $\pi(G_i) = \pi(G)$. Now $(G_1, \ldots, G_f, G, \ldots, G)$ is the unique element of $\sqbinom{\tilde V}{k} \cap \cH$ with $\phi(G_1, \ldots, G_f, G, \ldots, G) = G$. This establishes $\sqbinom{\tilde{V}}{k}\cap \cH\cong \cG$. 
    \end{proof}
	
	%Now consider an $r$-coloring $c:\cH \rightarrow [r]$. For an element $(Y_1,\ldots,Y_{N_1})\in \cY^{N_1}$, we can identify the $k$-spaces in $\cH$ by the projection of the coordinates over the map $\pi$. Indeed, for every $F \in \sqbinom{E}{k}$ and $1\leq i \leq N_1$, there exists a unique $k$-space $F_i\in \sqbinom{Y_i}{k}$ such that $\pi(F_i)=F$. That is
	%\begin{align*}
	    %\sqbinom{(Y_1,\ldots, Y_{N_1})}{k}\cap \cH=\left\{(F_1,\ldots,F_{N_1}):\: \tilde{\pi}(F_1,\ldots,F_{N_1})=F,\, F\in \sqbinom{E}{k}\right\}.
	%\end{align*}
	%Since $c$ is a coloring of certain $k$-subspaces of $(Y_1,\ldots,Y_{N_1})$, the coloring naturally defines a color pattern on $(Y_1,\ldots,Y_{N_1})$. Moreover, since every $k$-space of $\sqbinom{(Y_1,\ldots, Y_{N_1})}{k}\cap \cH$ is in a one-to-one correspondence with its projection over $\tilde{\pi}$, we can identify the color pattern by using the $k$-subspaces of $E$. Let $c_P:\cY^{N_1}\rightarrow [r]^{\sqbinom{E}{k}}$ be a coloring function from the elements of $\cY^{N_1}$ to all possible color patterns of $\sqbinom{E}{k}$ given by
	%\begin{align*}
	    %c_P(Y_1,\ldots,Y_{N_1})_F=c(F_1,\ldots,F_{N_1}),\quad F_i\in \sqbinom{Y_i}{k},\, \tilde{\pi}(F_1,\ldots,F_{N_1})=F.
	 %\end{align*}
	 
 	Now consider an $r$-coloring $c: \cH \rightarrow [r]$. Note that since every $G \in \cG$ is a $k$-subspace of some $Y \in \cY$, and the $k$-subspaces of each $(Y_1, \ldots, Y_{N_1}) \in \cY^{N_1}$ are all of the form $(G_1, \ldots, G_{N_1})$, with each $G_i \in \sqbinom{Y_i}{k} \subseteq \cG$, we have
 	\[\cH = \bigcup_{(Y_1, \ldots, Y_{N_1}) \in \cY^{N_1}} \sqbinom{(Y_1, \ldots, Y_{N_1})}{k}.\]
	Therefore, for each $(Y_1, \ldots, Y_{N_1}) \in \cY^{N_1}$, $c$ induces a color pattern on $\sqbinom{(Y_1, \ldots, Y_{N_1})}{k}$, and thus, in view of the isomorphism $\tilde \pi \vert_{(Y_1, \ldots, Y_{N_1})} : (Y_1, \ldots, Y_N) \to E$, also on $\sqbinom{E}{k}$.
 	Let $c_P : \cY^{N_1} \rightarrow [r]^{\sqbinom{E}{k}}$ be the coloring function from the elements of $\cY^{N_1}$ to the color patterns of $\sqbinom{E}{k}$ they induce.
	 
	By our choice of $N_1$ and the Hales--Jewett theorem, there exists a monochromatic combinatorial line $\cZ\subseteq \cY^{N_1}$ with respect to $c_P$. Let $d: \sqbinom{E}{k}\rightarrow [r]$ be the color pattern of each $(Z_1, \ldots, Z_{N_1}) \in \cZ$. Then for every $(Z_1, \ldots, Z_{N_1}) \in \cZ$ and every $(G_1, \ldots, G_{N_1}) \in \sqbinom{(Z_1, \ldots, Z_{N_1})}{k}$, we have
	\[c(G_1, \ldots, G_{N_1}) = d(\tilde \pi (G_1, \ldots, G_{N_1}));\]
	i.e. the coloring of $\cH$ within $\cZ$ is determined by its $\tilde \pi$-projection over $E$. Let $\tilde V$ and $\phi : \tilde V \to V$ be as in Proposition \ref{prop:prod}. Then $\tilde \cG := \phi^{-1}(\cG) = \sqbinom{\tilde V}{k} \cap \cH$ is an induced copy of $\cG$ in $\cH$ with $c(\tilde G) = d(\tilde \pi (\tilde G))$ for every $\tilde G \in \tilde{\cG}$.
	
	Since $d$ is an $r$-coloring of $\sqbinom{E}{k}$, by our choice of $N_0$, there exists a $n$-space $U \in \sqbinom{E}{n}$ such that $\sqbinom{U}{k}$ is monochromatic, say in color $\alpha$. In particular, this implies that $\cF_U \subseteq \sqbinom{U}{k}$ is monochromatic in $\alpha$. Now for $\tilde G \in \varphi^{-1}(\cF_{W_U}) \subseteq \tilde \cG$, we have
	\[\tilde \pi(\tilde G) = \pi(\phi(\tilde G)) \in \cF_U,\]
	and therefore $c(\tilde G) = d(\tilde \pi(\tilde G)) = \alpha$. Thus $\phi^{-1}(\cF_{W_U})$ is a monochromatic induced copy of $\cF$ in $\cH$.
	
\bibliography{literature}

\end{document}